\documentclass[11pt]{article}
\usepackage{anysize}\marginsize{3.5cm}{3.5cm}{1.3cm}{2cm}
\usepackage{amssymb}
\headsep=\baselineskip
\makeatletter\let\orig@item=\@item \def\@item[#1]{\orig@item[\rm #1]}\makeatother
\let\origcaption=\caption \renewcommand\caption[1]{\parbox{0.66\textwidth}{\origcaption{#1}}}

\newtheorem{satz}{Satz}[section]
\makeatletter \newcommand\numberwithin{\@addtoreset} \numberwithin{equation}{satz} \makeatother

\newtheorem{theorem}[satz]{Theorem}
\newtheorem{example}[satz]{Example}
\newtheorem{examples}[satz]{Examples}
\newtheorem{lemma}[satz]{Lemma}
\newtheorem{proposition}[satz]{Proposition}

\newtheorem{remarks}[satz]{Remarks}

\newtheorem{introtheorem}{Theorem}

\newenvironment{proof}[1][Proof]{\trivlist\item[\hskip\labelsep{\it #1.}]}{\hspace*{\fill}$\Box$\endtrivlist}

\newcommand\subjclass[1]{{\renewcommand\thefootnote{}\footnotetext{2000 \textit{Mathematics Subject Classification:} #1.}}}

\newcounter{comments}

\newcommand\inparen[1]{\textnormal{(}{#1}\textnormal{)}}
\newcommand\longto{\longrightarrow}
\renewcommand\ge{\geqslant}  
\renewcommand\le{\leqslant}  
\renewcommand\epsilon{\varepsilon}
\renewcommand\phi{\varphi}
\renewcommand\bar{\overline}
\renewcommand\hat{\widehat}

\renewcommand\O{\mathcal O}
\newcommand\HH[3]{\ensuremath{H^{#1}\!\left(#2,#3\right)}}
\newcommand\pd[2]{\frac{\partial #1}{\partial #2}} 
\newcommand\liste[3]{\mbox{$#1_{#2},\dots,#1_{#3}$}}
\newcommand\restr[1]{\big|_{#1}}
\newcommand\set[1]{\left\{\,#1\,\right\}}
\newcommand\with{\ \vrule\ }
\newcommand\midtext[1]{\quad\mbox{#1}\quad}

\newcommand\be{\begin{eqnarray*}}
\newcommand\ee{\end{eqnarray*}}
\newcommand\eqnref[1]{(\ref{#1})}
\newcommand\eqdef{\stackrel{\rm def}{=}}
\newcommand\eps{\varepsilon}

\renewcommand\P{\mathbb P}
\newcommand\N{\mathbb N}
\newcommand\Q{\mathbb Q}
\newcommand\R{\mathbb R}
\newcommand\Z{\mathbb Z}
\newcommand\C{\mathbb C}
\providecommand\engqq[1]{``#1''}
\newcommand\numequiv{\equiv_{\rm num}}
\newcommand\tensor{\otimes}
\newcommand\union{\cup}
\newcommand\maxmatrcols{10}
\newcommand\matr[1]{\left(\begin{array}{*{\maxmatrcols}{c}} #1 \end{array}\right)}
\newcommand\abs[1]{\left|#1\right|}
\newcommand\norm[1]{\left\|#1\right\|}
\newcommand\inverse{^{-1}}
\renewenvironment{cases}{\left\{\begin{array}{cll}}{\end{array}\right.}
\newcommand\tline{\noalign{\vskip0.4ex}\hline\noalign{\vskip0.65ex}}
\newcommand\converges[2]{\mathop{\longrightarrow}\limits_{n\to\infty}}
\newcommand\twoconditions[2]{_{\scriptstyle#1\atop\scriptstyle#2}}
\newcommand\interval[1]{\mathinner{#1}}

\newcommand\newop[2]{\newcommand#1{\mathop{\rm #2}\nolimits}}
\newop\mult{mult} 
\newop\NS{NS} 
\newop\Amp{Amp} 
\newop\Bl{Bl} 
\newop\End{End} 
\newop\Nef{Nef}
\newcommand\q[3]{\frac{#1\cdot #2}{\mult_{#3}#2}} 
\newcommand\Ncd{N_{c,d}} 
\newop\mod{\mkern 12mu mod} 
\newcommand\Nabcd{N_{a,b,c,d}}
\newcommand\gensubgroup[1]{\langle#1\rangle}

\newcommand\Llm{L_{\lambda,\mu}}
\newcommand\compact{\parskip=0cm\itemsep=0cm}

\begin{document}

\title{Seshadri constants on the self-product \\ of an elliptic curve}
\author{Thomas Bauer and Christoph Schulz}
\date{June 25, 2008}
\maketitle
\thispagestyle{empty}

\begin{abstract}
   The purpose of this paper is to study Seshadri constants on
   the self-product $E\times E$ of an elliptic curve $E$.
   We provide explicit formulas for
   computing the Seshadri constants of all ample line bundles on
   the surfaces considered. As an application, we obtain a good
   picture of the behaviour of the Seshadri function on the nef
   cone.
\end{abstract}

\section*{Introduction}

   For an ample line bundle $L$ on a smooth projective variety
   $X$ over the complex numbers,
   the \textit{Seshadri constant} of $L$ at $x\in X$ is
   by definition the real number
   $$
      \eps(L,x)=\sup\set{\eps>0\with f^*L-\eps E\mbox{ is nef}}
      \ ,
      \eqno\mbox{(*)}
   $$
   where $f:\Bl_x(X)\to X$ is the blow-up of $X$ at $x$ and $E$
   is
   the exceptional divisor over $x$
   (see \cite{Dem92} and \cite[Chapt.~5]{PAG}).
   Seshadri constants are invariants of ample line bundles that
   measure their local positivity at a given point. While they
   were originally intended as a means to produce sections of
   adjoint linear series, it soon became clear that they
   are interesting invariants quite in their own right. It has
   turned out, however, that it is quite difficult to
   determine explicit values except in obvious cases like
   projective space.

   There has been a considerable amount of work
   on Seshadri constants
   in recent years.
   One line of investigation concerns
   specific classes of
   surfaces, aiming
   for explicit bounds and, as far as possible,
   for explicit values of these subtle invariants
   (see for instance
   \cite{Bro06},
   \cite{Fue06},
   \cite{Knu07},
   \cite{Ros07},
   \cite{Tut03},
   \cite{Tut04}).
   Starting with \cite{Nak96}, Seshadri constants have been
   studied quite intensively on abelian varieties (see
   \cite{Laz97},
   \cite{Bau98a},
   \cite{BauSze98},
   \cite{Kon03},
   \cite{Deb04}).
   Here, by homogeneitiy, the Seshadri constant $\eps(L,x)$ is
   independent of the point $x$, so it is an invariant $\eps(L)$
   that is
   attached to every polarized abelian variety $(X,L)$.
   For abelian surfaces of Picard number one,
   the problem of finding explicit values for
   Seshadri constants
   was solved in
   \cite[Sect.~6]{Bau99}.
   In the present paper we attack the problem
   from the opposite end: we consider products
   of elliptic curves.
   While the task of determining Seshadri constants on a product
   $E_1\times E_2$ of two elliptic curves that are \textit{not
   isogenous} is an immediate exercise, the behaviour of Seshadri
   constants on the self-product $E\times E$ of \textit{one}
   elliptic curve turns out to be an interesting and non-trivial
   problem. The latter fact does perhaps not come as a surprise, as
   increasing the rank of the N\'eron-Severi group
   dramatically increases the choice of
   ample line bundles and curves that
   have to be taken into account in~(*).

   The problem naturally breaks up into two parts according to
   whether the elliptic curve has complex multiplication or not.
   In each case we are able to provide a complete picture.

\begin{introtheorem}
   Let $E$ be an elliptic curve without complex multiplication.
   On the abelian surface $X=E\times E$ denote by $F_1,F_2$
   the fibers of the projections and by $\Delta$
   the diagonal.
   Let $L=\O_X(b_1F_1+b_2F_2+b_3\Delta)$ be any ample line bundle
   on $X$, and take
   a permutation $(a_1,a_2,a_3)$ of
   $(b_1,b_2,b_3)$ satisfying $a_1\ge a_2\ge a_3$.

   Then $\eps(L)$ is the minimum of the following finitely many
   numbers:

   \begin{itemize}
   \item[(1)]
      $a_2+a_3$,
   \item[(2)]
      $\displaystyle\frac {a_2a_1^2+a_1a_2^2+a_3(a_1+a_2)^2}{\gcd(a_1,a_2)^2}$,
   \item[(3)]
      $\min\set{a_1d^2+a_2c^2+a_3(c+d)^2
      \with c ,d\in\N \mbox{ coprime}, \, c+d<\frac 1{\sqrt 2}(a_1+a_2)}$.
   \end{itemize}
\end{introtheorem}

   As an application, we obtain in
   Sect.~\ref{sect seshadri function}
   a good picture of the
   behaviour
   of the \textit{Seshadri function}
   $$
      \eps:\Nef(X)\longto\R,\quad L\mapsto\eps(L) \ .
   $$
   We find that this function is continuous on the nef cone of $X$,
   and that its
   cross-sections are piecewise linear (see
   Sect.~\ref{sect seshadri function} for examples).

   Our second main result
   concerns elliptic curves with complex multiplication. We focus
   on those two curves that admit an automorphism $\ne\pm 1$.
   We prove:

\begin{introtheorem}
   Let $E_1$ be the elliptic curve admitting the automorphism
   $\iota:[x]\mapsto [ix]$, i.e., $E_1=\C/(\Z+i\Z)$.
   On the abelian surface $X=E_1\times E_1$ denote by $F_1,F_2$
   the fibers of the projections, by $\Delta$
   the diagonal, and by $\Sigma$ the graph of $\iota$.
   Let
   $L=\O_X(a_1F_1+a_2F_2+a_3\Delta+a_4\Sigma)$ be any ample
   line bundle
   on $X$. Then
   \be
      && \eps(L) =
      \!\!\!\!\!\!
      \min\twoconditions{a,b,c,d\in\Z}{\abs a,\abs b,\abs c, \abs d\le B}
      \!\!\!\!\!\!
         \Big\{a_1(a^2+b^2)+a_2(c^2+d^2)
         \\[-3ex]
         && \hspace{10em} +a_3((a-c)^2+(b-d)^2)
         +a_4((a-d)^2+(b+c)^2)\Big\} \ ,
   \ee
   where
   $$
      B \eqdef\frac{8\max\set{\abs{a_1+a_3+a_4}^2,\abs{a_3}^2,\abs{a_4}^2,\abs{a_2+a_3+a_4}^2}}
         {a_1a_2+a_1a_3+a_1a_4+a_2a_3+a_2a_4+2a_3a_4} \ .
   $$
\end{introtheorem}

   A result of similar shape holds for the elliptic curve with
   automorphism $[x]\mapsto[e^{\pi i/3}x]$ (see Theorem
   \ref{SC theorem cx mult 2} for the precise statement).

   In our opinion it is a nice feature of both Theorem~1 and
   Theorem~2 that they allow the quick and effective
   computation of Seshadri
   constants
   just by taking the minimum of finitely many numbers. Concrete
   examples are shown in Tables
   \ref{table eps} and \ref{table eps cx mult}
   in Sections \ref{sect without cx mult} and \ref{sect cx mult}.

   This paper is organized as follows. We start in
   Sect.~\ref{sect SC} by
   very briefly providing the necessary background on Seshadri
   constants as well as an auxiliary result.
   In Sect.~\ref{sect without cx mult} we study
   abelian surfaces $E\times E$ where $E$ does not have complex
   multiplication.
   We apply these results in Sect.~\ref{sect seshadri function}
   in order to gain insight into the behaviour of
   the Seshadri function on the nef cone.
   Abelian surfaces $E\times E$ where $E$ has complex
   multiplication are studied
   in Sect.~\ref{sect cx mult}. The latter case is
   -- probably expectedly --
   technically harder and requires somewhat different methods.

\paragraph*{\it Convention.}
   We work throughout over the field of complex numbers.

\paragraph*{\it Acknowledgement.}
   This research was supported by DFG grant
   BA 1559/4-3.
   We have benefited from discussions with
   T.~Szemberg.

\section{Seshadri constants}\label{sect SC}

   Let $L$ be an ample line bundle on a smooth projective variety
   $X$, and let $\eps(L,x)$ be the Seshadri constant of $L$ at $x$
   as defined in the introduction.
   An alternative definition, which we will be using, is
   $$
      \eps(L,x)=\inf\set{\q LCx\with C\mbox{ irreducible curve
      passing through $x$}} \ .
   $$
   We mention that
   there is also a way to characterize Seshadri constants in
   terms of the separation of jets: One has
   $$
      \eps(L,x)=\limsup_{k\to\infty}\frac{s(kL,x)}k \ ,
   $$
   where $s(kL,x)$ is the maximal number of jets that the linear
   series $|kL|$ separates at $x$, i.e., the maximal integer $s$ such
   that the evaluation map
   $$
      \HH 0X{kL}\longto\HH 0X{kL\tensor\O_X/\mathfrak m_x^{s+1}}
   $$
   is onto.

   As a consequence of Kleiman's theorem, one has the upper bound
   $\eps(L,x)\le\sqrt{L^n}$, where $n=\dim(X)$.
   On abelian varieties, Seshadri constants enjoy the following
   additional properties:

   \begin{itemize}\compact
   \item
      By homogeneity, the Seshadri constant $\eps(L,x)$
      is independent of the point
      $x$. So it depends only on the line
      bundle, and we will write $\eps(L)$.
   \item
      One has the lower bound $\eps(L)\ge 1$, again as a
      consequence of homogeneity (see
      \cite[Example~5.3.10]{PAG}).
   \end{itemize}

   Consider now a smooth projective surface $X$.
   The following terminology turns out to be quite convenient:
   If $\eps(L,x)$ is smaller than the theoretical upper bound
   $\sqrt{L^2}$, then we will say that the Seshadri constant of
   $L$ at $x$ is \textit{submaximal}. If a curve $C$ satisfies the
   inequality
   $$
      \q LCx < \sqrt{L^2}
   $$
   at some point $x$,
   then we will call $C$ a \textit{submaximal curve} (\textit{for
   $L$ at $x$}). If
   $$
      \q LCx = \eps(L,x) \ ,
   $$
   then we will say that $C$ \textit{computes the Seshadri
   constant of $L$ at $x$}.
   One knows that if $\eps(L,x)$ is submaximal, then there must
   exist a curve that computes $\eps(L,x)$.
   Interestingly,
   by a result of
   Szemberg \cite[Proposition 1.8]{SzHabil}
   the number of submaximal curves for a given ample line bundle
   is bounded from above by the
   rank of the N\'eron-Severi group of $X$.

   We will make use of the following lemma from
   \cite[Sect.~5]{Bau99}.

\begin{lemma}\label{useful lemma}
   Let $X$ be a smooth projective surface, $L$ an ample line
   bundle on $X$, $x\in X$ and $\xi>0$. If there is a divisor
   $D\in|kL|$, $k\in\N$, such that
   $$
      \q LDx\le\xi\sqrt{L^2} \ ,
   $$
   then every irreducible curve with
   $$
      \q LCx<\frac1\xi\sqrt{L^2}
   $$
   is a component of $D$.
\end{lemma}

   As a somewhat surprising consequence,
   which has a crucial application in
   Sect.~\ref{sect cx mult}, an ample irreducible curve that
   is submaximal for \textit{some} ample line bundle in fact
   computes its own Seshadri constant:

\begin{proposition}\label{prop ample submaximal}
   Let $X$ be a smooth projective surface and $x\in X$. If $C$ is
   an irreducible ample
   curve that is submaximal at $x$ for some ample line
   bundle $L$, then $C$ computes $\eps(\O_X(C),x)$.
\end{proposition}

\begin{proof}
   From the index inequality and the assumption on $C$ we get
   $$
      \frac{\sqrt{L^2}\sqrt{C^2}}{\mult_x(C)}
      \le\q LCx<\sqrt{L^2} \ ,
   $$
   and hence
   $$
      \q {\O_X(C)}Cx < \sqrt{\O_X(C)^2} \ .
   $$
   As $C$ is irreducible, Lemma~\ref{useful lemma} (with $\xi=1$)
   implies that
   there cannot be any other submaximal curves for $\O_X(C)$ at
   $x$.
\end{proof}

   Note that the proposition remains true when \engqq{submaximal}
   is replaced by \engqq{weakly submaximal} (meaning that
   $L\cdot C/\mult_x(C)\le\sqrt{L^2}$ holds instead of the strict
   inequality).

\section{The case $E\times E$ without complex multiplication}\label{sect without cx mult}

   Let $E$ be an elliptic curve without complex multiplication.
   The abelian surface $X=E\times E$ is then of Picard number~3,
   and the N\'eron-Severi group is generated over $\Z$
   by the fibers $F_1$,
   $F_2$ of the projections $X\to E$ and the diagonal $\Delta$
   (see \cite[Sect.~2.7]{BLComplexTori}).

   A line bundle
   $$
      L=\O_X(a_1F_1+a_2F_2+a_3\Delta)
   $$
   is ample if and only if its
   the integer coefficients $a_1,a_2,a_2$ satisfy
   the following inequalities:
   \begin{equation}\label{ampleness conditions}
      a_1+a_2>0, \
      a_2+a_3>0, \
      a_3+a_1>0, \
      a_1a_2+a_2a_3+a_3a_1>0 \ .
   \end{equation}
   In fact, if $L$ is ample then its intersections with the
   curves $F_1,F_2,\Delta$, as well as its self-intersection must be
   positive, which shows that the inequalities are necessary.
   Conversely, if the inequalities are satisfied, then $L^2>0$ and
   the intersection of $L$ with the ample line bundle
   $\O_X(F_1+F_2)$ is positive, which implies that $L$ is ample
   (see \cite[4.3.2(b)]{LB}).

\begin{example}\label{warm-up example}\rm
   By way of warm-up let us
   consider an easy case first. Take an ample line bundle
   $L=\O_X(a_1F_1+a_2F_2+a_3\Delta)$, all of whose coefficients
   $a_i$ are non-negative. Let $D$ be the divisor
   $a_1F_1+a_2F_2+a_3\Delta$.
   For any irreducible curve $C$ passing through 0 and
   different from $F_1,F_2,\Delta$, we have
   $$
      L\cdot C=D\cdot C
      \ge\mult_0D\cdot\mult_0C
      \ge (a_1+a_2+a_3)\cdot\mult_0C \ ,
   $$
   and hence
   $$
      \q LC0\ge a_1+a_2+a_3 \ .
   $$
   On the other hand, as
   $L\cdot F_1=a_2+a_3$,
   $L\cdot F_2=a_1+a_3$, and
   $L\cdot \Delta=a_1+a_2$, we find that
   $$
      \eps(L)=\min\set{ a_1+a_2, a_2+a_3, a_3+a_1 } \ .
   $$
   So in this case one of the generators $F_1,F_2,\Delta$
   computes $\eps(L)$.

   Note that the argument in this example
   depends crucially on the fact that we
   know explicitly a suitable effective divisor $D$ in the linear
   series $|L|$. If we consider
   an ample line bundle like
   $\O_X(7F_1+6F_2-3\Delta)$ instead, no suitable effective
   divisor is apparent,
   and
   it is therefore not so clear how its
   Seshadri constant can be computed.
   We will return to this example in
   \ref{example eps}.
\end{example}

   Our purpose in this section is to determine  the Seshadri
   constants of all ample line bundles on $X$.
   The first point is to prove that all Seshadri constants on $X$
   are computed by elliptic curves (Theorem~\ref{computed by
   elliptic}). Based on this result we can then carry out the
   computation of the Seshadri constants (Theorem~\ref{SC
   theorem}).

\begin{theorem}\label{computed by elliptic}
   Let $E$ be an elliptic curve without complex multiplication,
   and let $X=E\times E$. For any ample line bundle $L$ on $X$,
   the
   Seshadri constant $\eps(L)$ is computed by an elliptic curve.
\end{theorem}

   So in particular, Seshadri constants on $X$ are always
   integers. For the proof of the theorem we need some
   preparation. To begin with, we determine all
   elliptic curves on $X$:

\begin{proposition}\label{all elliptic curves}
   \begin{itemize}
   \item[(i)]
      For every elliptic curve $N$ on $X$ that is not a translate
      of $F_1$, $F_2$ or $\Delta$ there exist coprime integers
      $c$ and $d$ such that one has the numerical equivalence
      $$
         N\numequiv c(c+d)F_1+d(c+d)F_2-cd\Delta \ .
      $$
   \item[(ii)]
      Conversely, for every pair of coprime integers $c$ and $d$
      the linear series
      $$
         |c(c+d)F_1+d(c+d)F_2-cd\Delta|
      $$
      consists of an elliptic curve.
   \end{itemize}
\end{proposition}

\begin{remarks}\label{Ncd remark}\rm
   (i)
   We will denote henceforth by $\Ncd$ the elliptic curve
   specified by Proposition~\ref{all elliptic curves}(ii).
   The curves $\Ncd$, along with the curves
   $F_1$, $F_2$, and $\Delta$, constitute then a complete system
   of representatives for the numerical classes of elliptic
   curves on $X$.

   (ii)
   If we drop in Proposition~\ref{all elliptic curves}(ii)
   the assumption that $c$ and $d$ be coprime,
   then even the curves $F_1$, $F_2$, and $\Delta$ occur among
   the $\Ncd$: Take $(c,d)=(1,0)$, $(0,1)$, and $(1,-1)$
   respectively.
   However, the system $|c(c+d)F_1+d(c+d)F_2-cd\Delta|$
   then represents non-reduced curves $\Ncd$ as well:
   If $m$ is the greatest common divisor of $c$ and
   $d$, then $\Ncd=mN$, where
   $N$ is an elliptic curve. It will be useful to
   take this broader point of view in the proof of
   \ref{computed by elliptic}.
\end{remarks}

\begin{proof}
   (i) Let $N$ be an elliptic curve as in the hypothesis. We can
   write
   $$
      N\numequiv a_1F_1+a_2F_2+a_3\Delta
   $$
   with integers $a_1$, $a_2$, $a_3$.
   Then
   \begin{equation}\label{N square}
      0 = N^2 = 2(a_1a_2+a_1a_3+a_2a_3) \ .
   \end{equation}
   From the hypothesis that $N$ is not numerically equivalent to
   any of the generators
   $F_1$, $F_2$, $\Delta$, it follows that none of the
   coefficients $a_i$ can be zero.
   In fact, if $a_1=0$, say, then \eqnref{N square} implies that
   $a_2=0$ or $a_3=0$, which gives $N\numequiv a_3\Delta$ or
   $N\numequiv a_2F_2$ respectively, and this in turn implies
   that $N\numequiv\Delta$ or $N\numequiv F_2$
   (see Lemma~\ref{lemma-indivisible} below).
   The same kind of reasoning yields
   $a_1+a_2\ne 0$.
   Equation \eqnref{N square} says then that
   $$
      -\frac{a_1a_2}{a_1+a_2}=a_3 \ ,
   $$
   hence $a_1+a_2$ divides $a_1a_2$.
   This implies by
   Lemma \ref{division lemma} below that
   there are integers $c$, $d$, $m$
   such that $c$ and $d$ are coprime and
   $$
      a_1=mc(c+d) \midtext{and} a_2=md(c+d) \ .
   $$
   So we have
   $$
      N\numequiv mc(c+d)F_1+md(c+d)F_2-mcd\Delta \ .
   $$
   As the numerical class of $N$ is indivisible (see
   Lemma~\ref{lemma-indivisible} below), we get
   $m=1$.

   (ii) Let $M$ be the line bundle
   $\O_X(c(c+d)F_1+d(c+d)F_2-cd\Delta)$. We find
   $$
      M^2=0 \midtext{and} M\cdot F_1=d^2>0 \ .
   $$
   It follows -- for instance from
   \cite[Lemma 2.4]{Bau95} -- that $h^0(M)>0$, and is is easy
   to see that, up to numerical equivalence, $M$
   is of
   the form $\O_X(mN)$, where $N$ is an elliptic curve and
   $m$ a positive integer. From the equations
   $$
      mN\cdot F_1=M\cdot F_1=d^2\midtext{and} mN\cdot F_2=M\cdot F_2=c^2
   $$
   we see then that $m=1$, since $c$ and $d$ are coprime.
\end{proof}

\begin{lemma}\label{division lemma}
   Let $a$ and $b$ be non-zero
   integers such that $a+b$ divides $ab$. Then
   there are integers $c$, $d$, and $m$, such that $c$ and $d$
   are coprime and
   $$
      a=mc(c+d), \  b=md(c+d) \ .
   $$
\end{lemma}

\begin{proof}
   Let $\ell$ be the greatest common divisor of $a$ and $b$, and
   let $c=a/\ell$ and $d=b/\ell$. Then $c$ and $d$ are coprime
   and we have
   $$
      a+b=\ell(c+d) \midtext{and} ab=\ell^2 cd \ .
   $$
   From the assumption that $a+b$ divides $ab$ we see that $c+d$
   divides $\ell cd$. Let $p$ be a prime divisor of $c+d$. Then
   $p$ also divides $\ell cd$. If $p$ were to divide $c$ or $d$,
   then, as a  prime divisor of $c+d$, it would divide
   both of them. But this cannot happen, as $c$ and $d$ are coprime.
   So none of the prime divisors of $c+d$ divides $c$ or $d$, and
   therefore $c+d$ divides $\ell$. Let now $m=\ell/(c+d)$. So we
   obtain
   \be
      a=\ell c=mc(c+d) \\
      b=\ell d=md(c+d)
   \ee
   as claimed.
\end{proof}

\begin{lemma}\label{lemma-indivisible}
   Let $X$ be an abelian surface and let $E\subset X$ be an
   elliptic curve. Then the numerical class of $E$ is
   indivisible. In other words, if
   $E\numequiv kD$ for some divisor $D$ and
   some integer $k>0$, then $k=1$.
\end{lemma}

\begin{proof}
   Fix an ample divisor $H$. Then $H\cdot D=\frac 1kH\cdot E>0$
   and $D^2=\frac1{k^2}E^2=0$, which implies that $\O_X(D)$ is
   effective (see e.g.~\cite[Lemma~2.4]{Bau95}).
   Then $\O_X(D)$, being effective and of zero self-intersection,
   must be numerically equivalent to a positive
   multiple $mE'$ of an elliptic curve $E'$. So we have
   $E\numequiv kD\numequiv kmE'$. A suitable translate of $E$ is
   therefore contained in the linear series $|kmE'|$. But this
   can only happen if $k=m=1$, because all elements of $|kmE'|$
   are reducible if $km>1$.
\end{proof}

   We turn now to the proof of Theorem~\ref{computed by
   elliptic}. The proof draws from two sources: First, we use
   a classical result from the geometry of numbers in
   order to show that every ample line bundle admits a submaximal
   elliptic curve. Secondly, we apply a result from \cite{BauSze98}
   in order to prove that no curve of genus $>1$ can be
   \engqq{more submaximal} than the elliptic ones.

   The result from the geometry of numbers that we will need is
   Hermite's classical theorem (see e.g. \cite[Sect.~II.3.2]{Cas97}):

\begin{theorem}[Hermite]\label{Hermite-Lagrange}
   Let $Q$ be a positive definite quadratic form of two
   variables,
   $$
      Q(x,y)=ax^2+2bxy+cy^2 \ ,
   $$
   and let $\delta=ac-b^2$ be its determinant.
   Then there is a non-zero point $p\in\Z^2$ such that
   $$
      Q(p)\le\sqrt{\frac 43\delta} \ .
   $$
\end{theorem}

\begin{proof}[Proof of Theorem \ref{computed by elliptic}]
   (i) Let $L=\O_X(a_1F_1+a_2F_2+a_3\Delta)$ be an ample line
   bundle on $X$. Its intersection number with the elliptic
   curve $\Ncd$ is a quadratic from in the variables $c$ and $d$:
   $$
      Q(c,d)\eqdef L\cdot \Ncd = \matr{c & d}\matr{a_2+a_3 & a_3 \\ a_3 & a_1+a_3} \matr{c \\ d} \ .
   $$
   It
   follows from the ampleness of $L$
   (using the inequalities (\ref{ampleness conditions}))
   that $Q$ is positive
   definite. The discriminant of $Q$ is
   $$
      \delta=a_1a_2+a_1a_3+a_2a_3=L^2/2  \ .
   $$
   Applying now
   Theorem~\ref{Hermite-Lagrange} we find that there
   is a non-zero point $(c,d)\in\Z^2$ such that
   $$
      Q(c,d)\le\sqrt{\frac 43\delta} \ .
   $$
   This implies that
   \begin{equation}\label{eq Hermite submax}
      L\cdot\Ncd\le\sqrt{\frac 43\delta}=\sqrt{\frac 23 L^2} \ .
   \end{equation}
   So in any event $\Ncd$ is a submaximal curve for $L$. ($\Ncd$
   is either
   an elliptic curve or a multiple of an elliptic curve, see
   Remark~\ref{Ncd remark}.b).
   So we
   have
   $$
      \eps(L)\le\sqrt{\frac 23 L^2} \ .
   $$

   (ii)
   To complete the proof we now
   show that there cannot be a curve of genus
   $>1$ computing $\eps(L)$. This can be seen as follows:
   It is a consequence of
   \cite[Theorem~A.1(b)]{BauSze98} -- or more precisely
   of the proof of
   that theorem -- that for an irreducible curve $C$ of
   arithmetic genus $>1$ on an
   abelian surface, one has
   $$
      \q LCx \ge \sqrt{\frac 78 L^2} \ .
   $$
   This inequality, together with
   (\ref{eq Hermite submax}), guarantees that
   one of the curves $\Ncd$
   computes $\eps(L)$.
\end{proof}

   Having established that all Seshadri constants on $X$ are
   computed by elliptic curves, we are now able to provide a
   complete picture of the Seshadri
   constants of all ample line bundles. In order to formulate the
   result in the most compact way, it is best to keep in mind
   the following easy lemma.

\begin{lemma}\label{permutations}
   Let $L=\O_X(a_1F_1+a_2F_2+a_3\Delta)$ be an ample line bundle,
   let $\pi$ be a permutation of the numbers 1,2,3, and let
   $L^\pi=\O_X(a_{\pi(1)}F_1+a_{\pi(2)}F_2+a_{\pi(3)}\Delta)$ be
   the
   line bundle with permuted coefficients. Then $L^\pi$ is ample
   as well, and
   $$
      \eps(L^\pi)=\eps(L) \ .
   $$
\end{lemma}

\begin{proof}
   The intersection matrix of $(F_1,F_2,\Delta)$ is
   $$
      \matr{0 & 1 & 1 \\ 1 & 0 & 1 \\ 1 & 1 & 0} \ ,
   $$
   and any permutation of the triplet $(F_1,F_2,\Delta)$ has the
   same intersection matrix. This implies that $(L^\pi)^2=L^2$,
   and, if
   the linear series $|b_1F_1+b_1F_2+b_3\Delta|$ represents an
   elliptic curve, then the
   linear series with
   permuted coefficients also represents an elliptic curve~$N^\pi$.
   This curve $N^\pi$
   satisfies
   $$
      L^\pi\cdot N^\pi=L\cdot N \ ,
   $$
   so that if $N$ computes $\eps(L)$, then $N^\pi$ computes
   $\eps(L^\pi)$.
\end{proof}

   Our result can then be stated as follows, proving Theorem~1
   from the introduction.

\begin{theorem}\label{SC theorem}
   Let $E$ be an elliptic curve without complex multiplication
   and let
   Let $L=\O_X(a_1F_1+a_2F_2+a_3\Delta)$ be any ample line bundle
   on the abelian surface $X=E\times E$.
   Assume that
   $$
      a_1\ge a_2\ge a_3
   $$
   \inparen{which in view of Lemma \ref{permutations} means no loss in
   generality}.
   Then $\eps(L)$ is the minimum of the following
   numbers:

   \begin{itemize}
   \item[(1)]
      $a_2+a_3$,
   \item[(2)]
      $\displaystyle\frac {a_2a_1^2+a_1a_2^2+a_3(a_1+a_2)^2}{\gcd(a_1,a_2)^2}$,
   \item[(3)]
      $\min\set{a_1d^2+a_2c^2+a_3(c+d)^2
      \with c ,d\in\N \mbox{ coprime}, \, c+d<\frac 1{\sqrt 2}(a_1+a_2)}$.
   \end{itemize}
\end{theorem}

\begin{proof}
   We know by Theorem~\ref{computed by elliptic} that
   $\eps(L)$ is in any event
   computed by an elliptic curve. So $\eps(L)$ is the minimal
   degree $L\cdot N$, where $N$ runs through all elliptic curves
   on $X$, i.e.,
   $$
      \eps(L)=\min(
      \set{L\cdot F_1, L\cdot F_2, L\cdot \Delta}
      \union
      \set{L\cdot \Ncd\with\mbox{$c$ and $d$ coprime integers}}
      )
      \ .
   $$
   Expression (1) in the statement accounts for the curves $F_1$,
   $F_2$, and $\Delta$. The point now is to explicitly
   restrict the range of
   elliptic curves $\Ncd$ that have to be taken into account.

   Under our assumption that $a_1\ge a_2\ge a_3$ we see from the
   ampleness conditions (\ref{ampleness conditions}) that
   $a_1$ and $a_2$ must both be positive.
   We now determine when the elliptic curve
   $\Ncd$ is
   submaximal for $L$, i.e., when $L\cdot\Ncd<\sqrt{L^2}$ holds.
   In terms of coefficients this condition
   evaluates to the inequality
   $$
      (a_2+a_3)c^2+2a_3cd+(a_1+a_3)d^2 < \sqrt{2(a_1a_2+a_1a_3+a_2a_3)} \ .
   $$
   A calculation shows that the latter condition can equivalently
   be expressed as
   \be
      \lefteqn{\left(a_3(c+d)^2+\frac{(a_1d^2+a_2c^2)(c+d)^2-(a_1+a_2)}{(c+d)^2}\right)^2}  \\
      && < \frac 1{(c+d)^4}\left((a_1+a_2)^2-2(a_1d-a_2c)^2(c+d)^2\right) \ .
   \ee
   The crucial point is now that for
   this inequality to be satisfied -- given $a_1,a_2,a_3$ --
   it is necessary to have
   \begin{equation}\label{cd inequality}
      (a_1+a_2)^2 > 2(a_1d-a_2c)^2(c+d)^2 \ ,
   \end{equation}
   and this inequality narrows down the potential submaximal curves
   $\Ncd$
   to a finite set: First, we see that $c$ and $d$ must be
   both positive or both negative, as otherwise
   $$
      2(a_1d-a_2c)^2(c+d)^2\ge
      2(a_1\abs d-a_2\abs c)^2(c+d)^2\ge (a_1+a_2)^2 \ .
   $$
   We may therefore assume $c>0$ and $d>0$. (Note that
   $\Ncd=N_{-c,-d}$.)
   Furthermore, (\ref{cd inequality}) implies that
   $$
      \frac{a_2}{a_1}=\frac dc \midtext{or}
      (c+d)^2 < \frac 12(a_1+a_2)^2 \ .
   $$
   As $c$ and $d$ are coprime, the first case applies only to one
   elliptic curve, namely to
   $$
      N_{{a_1}/{\gcd(a_1,a_2)},\,{a_2}/{\gcd(a_1,a_2)}} \ ,
   $$
   which is taken account for by expression (2) of the theorem.
   The second case yields the range expressed in (3).
\end{proof}

\begin{remarks}\rm
   (i) Theorem \ref{SC theorem} shows that it is quick and easy
   to compute
   $\eps(L)$ from the coefficients $a_1,a_2,a_3$ of $L$:
   All one needs is to take the minimum of finitely many numbers.

   (ii) Note that there would be no harm if we extended the
   minimum in (3) over \textit{all} pairs of positive integers
   $c$ and $d$ with $c+d<\frac 1{\sqrt 2}(a_1+a_2)$, whether or
   not they are coprime. From a computational point of view it
   may in fact be more efficient to do so, forgoing any
   coprimality tests.
\end{remarks}

   Theorem \ref{SC theorem} allows not only to compute Seshadri
   constants, but it also yields all submaximal curves as the
   following examples illustrate.
   Table \ref{table eps} gives further concrete examples.

\begin{examples}\label{example eps}\rm
   (i) Consider the ample bundle $L=\O_X(7F_1+6F_2-3\Delta)$ that
   was mentioned briefly at the end of Example \ref{warm-up example}.
   Applying Theorem \ref{SC theorem} we find that
   $N_{1,1}$ calculates $\eps(L)=1$, and this is the only
   submaximal curve for $L$.

   (ii) As for an example at the other extreme:
   The ample bundle $L=\O_X(33F_1+9F_2-7\Delta)$ admits three
   submaximal curves, $F_1,N_{3,1},N_{4,1}$. All three of them
   compute $\eps(L)$ in this case. This is a case
   where the maximal possible number of submaximal curves
   occurs.
\end{examples}

\begin{table}
   \centering\footnotesize

\begin{tabular}{ccc@{\qquad}cccll}
$a_1$ & $a_2$ & $a_3$ & $L^2$ & $\sqrt{\frac23L^2}$ & $\eps(L)$ & curves computing $\eps(L)$ & weakly submaximal \\ \tline
3 & 2 & $-1$ & 2 & $\approx 1,15$ & 1 & $ F_1, N_{ 1,1 } $ & $F_1, N_{1 ,1 } $ \\
3 & 3 & $-1$ & 6 & 2 & 2 & $ F_1, F_2, N_{ 1,1 } $ & $F_1, F_2, N_{1 ,1 } $ \\
4 & 3 & $-1$ & 10 & $\approx 2,58$ & 2 & $ F_1 $ & $ F_1, N_{ 1,1 } $ \\
5 & 3 & $-1$ & 14 & $\approx 3,06$ & 2 & $ F_1 $ & $ F_1 $ \\
5 & 4 & $-2$ & 4 & $\approx 1,63$ & 1 & $ N_{ 1,1 } $ & $ F_1, N_{ 1,1 } $ \\
7 & 4 & $-2$ & 12 &  $\approx 2,83$ & 2 & $ F_1 $ & $ F_1, N_{ 1,1 } $ \\
7 & 6 & $-3$ & 6 & 2 & 1 & $ N_{ 1,1 } $ & $ N_{ 1,1 } $ \\
10 & 7 & $-4$ & 4 & $\approx 1,63$ & 1 & $ N_{ 1,1 } $ & $ N_{ 1,1 }, N_{ 2,1 } $ \\
12 & 9 & $-5$ & 6 & 2 & 1 & $ N_{ 1,1 } $ & $ N_{ 1,1 } $ \\
17 & 10 & $-6$ & 16 & $\approx 3,27$ & 3 & $ N_{ 1,1 }, N_{ 2,1 } $ & $ F_1, N_{ 1,1 }, N_{ 2,1 } $ \\
20 & 11 & $-7$ & 6 & 2 & 1 &  $ N_{ 2,1 } $ & $ N_{ 2,1 } $ \\
32 & 9 & $-7$ & 2 & $\approx 1,15$ & 1 & $ N_{ 3,1 }, N_{ 4,1 } $ & $ N_{ 3,1 }, N_{ 4,1 } $ \\
33 & 9 & $-7$ & 6 & 2 & 2 & $ F_1, N_{ 3,1 }, N_{ 4,1 } $ & $ F_1, N_{ 3,1 }, N_{ 4,1 } $ \\
34 & 9 & $-7$ & 10 & $\approx 2,58$ & 2 & $ F_1 $ & $ F_1, N_{ 3,1 }, N_{ 4,1 } $ \\
26 & 14 & $-9$ & 8 & $\approx 2,31$ & 1 & $ N_{ 2,1 } $ & $ N_{ 2,1 } $ \\
73 & 13 & $-11$ & 6 & 2 & 2 & $ F_1, N_{ 5,1 }, N_{ 6,1 } $ & $ F_1, N_{ 5,1 }, N_{ 6,1 } $ \\
54 & 14 & $-11$ & 16 & $\approx 3,27$ & 3 & $ F_1, N_{ 4,1 } $ & $ F_1, N_{ 3,1 }, N_{ 4,1 } $ \\
45 & 15 & $-11$ & 30 & $\approx 4,47$ & 4 & $ F_1, N_{ 3,1 } $ & $ F_1, N_{ 3,1 } $ \\
36 & 16 & $-11$ & 8 & $\approx 2,31$ & 1 & $ N_{ 2,1 } $ & $ N_{ 2,1 } $ \\
32 & 17 & $-11$ & 10 & $\approx 2,58$ & 1 & $ N_{ 2,1 } $ & $ N_{ 2,1 } $ \\
52 & 30 & $-19$ & 4 & $\approx 1,63$ & 1 & $ N_{ 2,1 } $ & $ N_{ 2,1 }, N_{ 5,3 } $
\\ \tline
\end{tabular}

   \caption{\label{table eps}
      Seshadri constants of the line bundles
      $L=O_X(a_1F_1+a_2F_2+a_3\Delta)$
      on $X=E\times E$.
      The last column lists all elliptic curves $C$ such that $L\cdot C\le \sqrt{L^2}$.
      }
\end{table}

\section{The Seshadri function on the nef cone}\label{sect seshadri function}

   Our purpose now is to apply the results of the
   previous section in order to gain insight into the
   behaviour of the Seshadri function on the nef cone of $E\times
   E$.

   Consider first an arbitrary smooth projective variety $Y$. The
   definition of Seshadri constants extends immediately to
   ample (or nef) $\Q$-divisors, and also to ample (or nef)
   $\R$-divisors (using either
   definition (*) from the introduction or the alternative
   characterization at the beginning of Sect.~\ref{sect SC}).
   Further, the definition clearly extends to nef divisors.
   We get thus for fixed $y\in Y$ a function
   $$
      \eps_y:\Nef(Y)\longto\R,\quad L\longmapsto\eps(L,y)
   $$
   on the nef cone of $Y$,
   which we will refer to as the \textit{Seshadri function} at
   $y$.

   Considering now an abelian variety $A$,
   we obtain a
   function
   $$
      \eps:\Nef(A)\longto\R,\quad L\longmapsto\eps(L) \
   $$
   that is independent of the point.
   Our first observation is:

\begin{proposition}
   Let $A$ be an abelian variety. Then
   the Seshadri function $\eps$ is concave and continuous.
\end{proposition}

   Note that this result (and the subsequent proof) remains valid
   more generally on homogeneous varieties.

\begin{proof}
   The concavity is immediate, as
   both the equality
   $\eps(\lambda L)=\lambda\eps(L)$ for $\lambda\ge0$ and
   the inequality
   $$
      \eps(L+M)\ge\eps(L)+\eps(M)
   $$
   follow immediately from the definition.
   The continuity in the interior of $\Nef(A)$ is then a
   consequence of concavity. Consider then an $\R$-line bundle
   $L$ on the boundary of the nef cone. According to the
   Nakai criterion for $\R$-divisors
   (see \cite[Theorem~2.3.18]{PAG}),
   there is a subvariety $V\subset A$ such that
   $L^d\cdot V=0$, where $d=\dim V$.
   Therefore, as a suitable translate of $V$ passes through any
   given point $x\in A$,
   $$
      0\le\eps(L)\le\sqrt[d]{\q{L^d}Vx}=0 \ ,
   $$
   and hence $\eps(L)=0$. Let now $(L_n)_{n\ge 1}$ be a sequence of
   $\R$-line bundles in $\Nef(A)$ converging to $L$.
   As the intersection product is continuous, we
   obtain
   $$
      0\le\eps(L_n)\le\sqrt[d]{\q{L_n^d}Vx}
      \converges n\infty
      \sqrt[d]{\q {L^d}Vx}=0=\eps(L) \ ,
   $$
   hence $\eps(L_n)\to\eps(L)$, as claimed.
\end{proof}

   Consider now $X=E\times E$, the self-product of an elliptic
   curve $E$ without complex multiplication,
   as in the preceding section.
   We wish to study the behaviour of its Seshadri function
   $\eps:\Nef(X)\to\R$.

   Let $L=\O_X(a_1F_1+a_2F_2+a_3\Delta)$ be an (integral)
   nef line bundle.
   We may assume $a_1\ge a_2\ge a_3$, and even $a_1>0$ if $L$ is
   not the trivial bundle.
   Writing then
   $$
      L=a_1\cdot \Llm,\quad
      \Llm=\O_X(F_1+\lambda F_2-\mu\Delta)
   $$
   with $\lambda=a_2/a_1$ and $\mu=-a_3/a_1$, it is enough to
   determine the Seshadri constants of the bundles $\Llm$.
   These are nef in the range
   $$
      \lambda\in[0,1],\quad
      \mu\in{]\!-\infty,\frac{\lambda}{1+\lambda}]} \ .
   $$
   The following statements are quickly verified:

   \begin{itemize}\compact
   \item[(i)]
      For $\mu\in\interval{]\!-\infty,-1]}$, the curve $\Delta$ computes
      $\eps(\Llm)=1+\lambda$.
   \item[(ii)]
      For $\mu\in\interval{]\!-1,0]}$, the curve $F_1$ computes
      $\eps(\Llm)=\lambda-\mu$.
   \end{itemize}
   As a consequence of Theorem~\ref{SC theorem} we now show:

\begin{proposition}\label{piecewise}
   For fixed rational $\lambda\in[0,1]$, the function
   $$
      {]\!-\infty,\frac{\lambda}{1+\lambda}]}\longto\R,\quad
      \mu\longmapsto\eps(\Llm)
   $$
   is a piecewise affine-linear function
   and has only finitely many
   affine-linear pieces.
\end{proposition}

\begin{proof}
   We may assume $\lambda>0$, so that $\Llm$ is ample.
   According to (the proof of) Theorem~\ref{SC theorem}, only
   finitely many of the
   elliptic curves $\Ncd$ can be submaximal for any
   of the line bundles $\Llm$, when $\lambda$ is fixed and $\mu$
   varies in the ample
   range $-\infty<\mu<\lambda/(1+\lambda)$.
   In fact, Condition
   \eqnref{cd inequality},
   which is necessary for submaximality,
   is equivalent to
   $$
      (1+\lambda)^2>2(d-\lambda c)^2(c+d)^2 \ ,
   $$
   and hence it is independent of $\mu$.
   Denoting the potential submaximal curves by $\liste N1k$,
   the Seshadri function in the statement of the proposition
   is then the pointwise minimum of
   finitely many affine-linear functions:
   \begin{equation}\label{minimum}
      \eps(\Llm)=\min_{i=1}^k \Llm\cdot N_i \ .
   \end{equation}
   At the upper boundary $\mu_\infty=\lambda/(1+\lambda)$ of the
   ample range, $\Llm$ is numerically equivalent to a multiple of
   an elliptic curve $\Ncd$, and hence
   $\eps(L_{\lambda,\mu_\infty})=0$.
\end{proof}

   The behaviour of the Seshadri function described by
   Proposition~\ref{piecewise} is displayed in
   Fig.~\ref{seshadri function}.
   We now illustrate the situation by considering concrete
   examples.

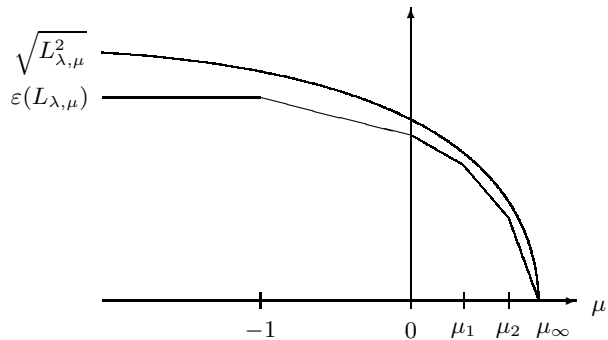
\begin{figure}
   \centering\footnotesize

\unitlength 1mm 
\linethickness{0.4pt}
\ifx\plotpoint\undefined\newsavebox{\plotpoint}\fi 
\begin{picture}(83,55)(0,0)
\put(38,7){\line(0,-1){2}}
\put(38,2){\makebox(0,0)[cc]{$-1$}}
\put(58,6){\vector(0,1){39}}
\put(58,2){\makebox(0,0)[cc]{$0$}}
\put(38,33){\line(4,-1){20}}
\multiput(58,28)(.058823529,-.033613445){119}{\line(1,0){.058823529}}
\multiput(65,24)(.033707865,-.039325843){178}{\line(0,-1){.039325843}}
\multiput(71,17)(.033613445,-.092436975){119}{\line(0,-1){.092436975}}
\put(71,7){\line(0,-1){2}}
\put(65,7){\line(0,-1){2}}
\put(58,7){\line(0,-1){2}}
\put(65,2){\makebox(0,0)[cc]{$\mu_1$}}
\put(71,2){\makebox(0,0)[cc]{$\mu_2$}}
\put(77,2){\makebox(0,0)[cc]{$\mu_\infty$}}
\qbezier(75,6)(74,35.5)(17,39)
\put(38,33){\line(-1,0){21}}
\put(17,6){\vector(1,0){63}}
\put(83,5){\makebox(0,0)[cc]{$\mu$}}
\put(10,33){\makebox(0,0)[cc]{$\varepsilon(L_{\lambda,\mu})$}}
\put(10,39){\makebox(0,0)[cc]{$\sqrt{L_{\lambda, \mu}^2}$}}
\end{picture}

   \caption{\label{seshadri function}
      Piecewise linear behaviour of the Seshadri function on
      a cross-section of the ample cone of the surface
      $X=E\times E$. (The point $\mu_\infty$ is the upper boundary
      $\frac\lambda{1+\lambda}$
      of the nef
      range.)
      }
\end{figure}

\begin{example}\rm
   We consider the line bundles
   $L_{\frac 1n,\mu}$ for a fixed integer $n\ge 1$. The nef range
   for $\mu$ is then $-\infty<\mu\le\frac{1}{n+1}$.
   One shows now that the only curves that matter in the minimum
   in \eqnref{minimum} are $\Delta,F_1$, and $N_{n,1}$.
   (If $n=2$, then
   the curve
   $N_{1,1}$ is also submaximal, but this curve
   turns out to be irrelevant when taking the
   minimum.)
   One can then determine the Seshadri function:
   $$
      \renewcommand\arraystretch{1.2}
      \eps(L_{\frac1n,\mu})=
      \begin{cases}
         1+\frac1n & \mbox{if } \mu\le-1 & \mbox{($\Delta$ computes $\eps$)} \\
         \frac1n-\mu & \mbox{if } \mu\in[-1,\frac{n^2+n-1}{n^2(n+2)}] & \mbox{($F_1$ computes $\eps$)} \\
         1+n-(n+1)^2\mu & \mbox{if } \mu\in[\frac{n^2+n-1}{n^2(n+2)},\frac1{n+1}] & \mbox{($N_{1,1}$ computes $\eps$)} \\
      \end{cases}
   $$
\end{example}

   For other values of $\lambda$, the number of elliptic curves $\Ncd$
   that have to be taken into account can become larger.
   We conclude with a somewhat more intricate example, which is
   intended
   to illustrate this point.

\begin{example}\rm
   We consider the line bundles
   $L_{\frac8{11},\mu}$.
   Among the curves $\Ncd$ the
   potential submaximal curves are
   $N_{1,1},N_{2,1},N_{3,2},N_{4,3},N_{7,5},N_{11,8}$. By carrying out the
   necessary computations one gets
   $$
      \renewcommand\arraystretch{1.2}
      \eps(L_{\frac8{11},\mu})=
      \begin{cases}
         \frac{19}{11} & \mbox{if } \mu\le -1 & \mbox{($\Delta$ computes $\eps$)} \\
         \frac{8}{11}-\mu & \mbox{if } \mu\in[-1,\frac13] & \mbox{($F_1$ computes $\eps$)} \\
         \frac{19}{11}-4\mu & \mbox{if } \mu\in [\frac13 ,\frac{97}{231}] & \mbox{($N_{1,1}$ computes $\eps$)} \\
         \frac{116}{11}-25\mu & \mbox{if } \mu\in [\frac{97}{231} ,\frac{37}{88}] & \mbox{($N_{3,2}$ computes $\eps$)} \\
         \frac{227}{11}-49\mu & \mbox{if } \mu\in [\frac{37}{88} ,\frac{1445}{3432}] & \mbox{($N_{4,3}$ computes $\eps$)} \\
         {152}-361\mu & \mbox{if } \mu\in[\frac{1445}{3432},\frac{8}{19} ] & \mbox{($N_{11,8}$ computes $\eps$)} \ . \\
      \end{cases}
   $$
   ($N_{2,1}$ and $N_{7,5}$ turn out to be irrelevant when taking
   the minimum.)
\end{example}

\section{The case $E\times E$ with complex multiplication}\label{sect cx mult}

   In this section we consider abelian surfaces $E\times E$
   where $E$ has complex multiplication. We will focus on the
   elliptic curves admitting an automorphism $\ne\pm 1$:
   $$
      E_1=\C / \Z+i\Z \midtext{and} E_2=\C / \Z+e^{\pi i/3}\Z \ .
   $$
   We will study first $E_1\times E_1$.

\subsection{Complex multiplication by $i$}

   The N\'eron-Severi group of $E_1\times E_1$ is of rank four, with
   generators
   $$
      F_1, F_2, \Delta, \Sigma ,
   $$
   where $F_1,F_2$ are the fibers of the projections, $\Delta$
   is
   the diagonal, and $\Sigma$ is the graph of the automorphism
   $$
      \iota:E_1\longto E_1, [x]\longmapsto [ix] \ .
   $$
   (see \cite[Sect.~2.7]{BLComplexTori}).

   Note that $\iota$ has exactly two fixed-points: $[0]$ and
   $[\frac{1+i}2]$. Therefore we have $\Delta\cdot\Sigma=2$.
   As for the remaining intersection numbers, we get
   $$
      F_1^2=F_2^2=\Delta^2=\Sigma^2=0
   $$
   and
   $$
      F_1\cdot F_2=F_1\cdot\Delta=F_2\cdot\Delta=F_1\cdot\Sigma=F_2\cdot\Sigma=1
      \ .
   $$

   A line bundle $L=O_X(a_1F_1+a_2F_2+a_3\Delta+a_4\Sigma)$ is
   ample if and only if its self-intersection as well as its
   intersection with the curves $F_1,F_2,\Delta,\Sigma$ are
   positive. (This follows in the same way as indicated after
   \eqnref{ampleness conditions} in the rank three case.)
   So $L$ is ample if and only if
   \be
        a_2+a_3+a_4&>&0 \\
        a_1+a_3+a_4&>&0 \\
        a_1+a_2+2a_4&>&0 \\
        a_2+a_2+2a_3&>&0 \\
        a_1a_2+a_1a_3+a_1a_4+a_2a_3+a_2a_4+2a_3a_4 &>& 0 \ .
   \ee

   The first step in this section is to prove an analogue of
   Theorem~\ref{computed by elliptic} to the effect that all
   Seshadri constants are computed by elliptic curves.
   To this end we will need to know all elliptic curves on
   $E_1\times E_1$.
   As a parametrization of all elliptic curves as in
   Sect.~\ref{sect without cx mult} seems difficult, we will
   make use of the following result instead:

\begin{lemma}[Hayashida-Nishi \cite{HN65}]\label{lemma HN}
   Let $E$ be an elliptic curve. Then
   for every elliptic curve $N$ on
   $E\times E$ there are endomorphisms $\sigma_1,\sigma_2$ of $E$
   such that $N$ is a translate of the image of the map
   $$
      E\longto E\times E, \ x\longmapsto(\sigma_1(x),\sigma_2(x)) \ .
   $$
\end{lemma}

   Let now $N$ be an elliptic curve on $X=E_1\times E_1$. As
   $\End(E_1)=\Z+\iota\Z$,
   Lemma \ref{lemma HN} says that
   there are integers $a,b,c,d$ such that
   $N$ is a translate of the curve
   $$
      N_{a,b,c,d}\eqdef\set{(ax+b\iota(x),cx+d\iota(x))\with x\in
      E_1} \ .
   $$
   We may assume here that $a,b,c,d$ are coprime, because a common
   factor would just mean that the map $(\sigma_1,\sigma_2)$ is
   composed with a multiplication map.

   We determine next the intersection numbers of
   $\Nabcd$ with the
   generators of the N\'eron-Severi group.
   As $F_1$ and $\Nabcd$ intersect transversely, we have
   \begin{equation}\label{intersection N F1}
      \Nabcd\cdot F_1=\#(\Nabcd\cap F_1)=
      \frac{\#\set{x\in F_1\with ax+b\iota x=0}}{\deg\sigma} \ .
   \end{equation}
   where $\sigma:E_1\to\Nabcd$ is the map
   $x\mapsto(ax+b\iota x,cx+d\iota x)$.

   In the next two lemmas we will evaluate the expression on the
   RHS of
   \eqnref{intersection N F1}.

\begin{lemma}\label{number of solutions}
   For integers $a$ and $b$, not both of them zero, the equation
   \begin{equation}\label{eqn on E1}
      ax+b\iota x=0
   \end{equation}
   has exactly $a^2+b^2$ solutions $x\in E_1$.
\end{lemma}

\begin{proof}
   We may assume that both $a$ and $b$ are non-zero, the
   assertion being clear otherwise.
   Let $\ell=a^2+b^2$, and
   consider first the case that $a$ and $b$ are coprime.
   Suppose that $x$ is a solution of \eqnref{eqn on E1}.
   By subtracting  the two equations that are obtained from
   \eqnref{eqn on E1} by multiplication with $a$ and $b$
   respectively, we see that $x$ is necessarily an
   $\ell$-division point on $E_1$. Now, an $\ell$-division point
   $$
      x=\left[\frac m\ell+i\frac n\ell\right], \qquad 0\le m,n < \ell,
   $$
   solves \eqnref{eqn on E1} if any only if $\ell$ is a divisor
   of both $am-bn$ and $an+bm$. Given an integer
   $m\in\{0,\dots,\ell-1\}$, there is a unique integer
   $n\in\{0,\dots,\ell-1\}$ such that these two divisibility
   conditions are satisfied (since $a$ and $b$ are invertible
   modulo $\ell$). So there are $\ell$ distinct
   solutions $x\in E_1$.

   Taking now general $a$ and $b$, let $d=\gcd(a,b)$ and write
   $a=da'$,
   and $b=db'$. By what we have shown so far, the equation
   $$
      a'(dx)+b'\iota(dx)=0
   $$
   admits exactly $a'^2+b'^2$ solutions for $dx$. As
   multiplication by $d$ is a map of degree $d^2$, we obtain
   $d^2(a'^2+b'^2)=a^2+b^2$ solutions for $x$, and this completes
   the proof.
\end{proof}

   We now determine the degree of the map
   $\sigma=(\sigma_1,\sigma_2):E_1\to\Nabcd$.
   For this, and in fact
   for the remainder of this section we will use the abbreviation
   \begin{equation}\label{gcd}
      D \eqdef \gcd(a^2+b^2,c^2+d^2,ac+bd,ad-bc) \ .
   \end{equation}

\begin{lemma}\label{degree of sigma}
   The map $\sigma$ is
   of degree $D$.
\end{lemma}

\begin{proof}
   We need to determine the number of elements in the kernel of
   $\sigma$. So suppose that $x$ is a point in $E_1$ with
   \begin{equation}\label{kernel sigma}
      ax+b\iota x=cx+d\iota x=0 \ .
   \end{equation}
   As in the proof of Lemma \ref{number of solutions} it follows
   that $x$ is both an $(a^2+b^2)$-division point
   and a $(c^2+d^2)$-division point.
   We see from the equation
   $d\iota(ax+b\iota x)=(-ac-bd)x$ that $x$ is also an
   $(ac+bd)$-division point, and in the same manner that it is
   also a $(ad-bc)$-division point. So we infer that $x$ is a
   $D$-divison point. Conversely, a $D$-division point
   $
      x=\left[\frac mD+i\frac nD\right]
   $
   satisfies the equations \eqnref{kernel sigma} if and only if
   the following congruences are fulfilled:
   \be
      && am-bn\equiv 0 \mod D \\
      && bm+an\equiv 0 \mod D \\
      && cm-dn\equiv 0 \mod D \\
      && dm+cn\equiv 0 \mod D \ .
   \ee
   The proof is now completed by invoking
   Lemma \ref{solutions of congruences} (in the appendix), which
   states
   that this systems admits exactly $D$ solutions.
\end{proof}

   The preceding lemmas now allow us to determine the required
   intersection numbers:

\begin{proposition}\label{Nabcd intersection numbers}
   We have
   $$
   \begin{array}{ll}
      \displaystyle\Nabcd\cdot F_1=\frac{a^2+b^2}D &
      \displaystyle\Nabcd\cdot F_2=\frac{c^2+d^2}D
      \\[\bigskipamount]
      \displaystyle\Nabcd\cdot \Delta=\frac{(a-c)^2+(b-d)^2}D  \qquad &
      \displaystyle\Nabcd\cdot \Gamma=\frac{(a-d)^2+(b+c)^2}D
   \end{array}
   $$
\end{proposition}

\begin{proof}
   In view of lemmas \ref{number of solutions}
   and \ref{degree of sigma}
   the first assertion follows using
   \eqnref{intersection N F1}.
   The proof of the remaining assertions is analogous.
\end{proof}

   Fix now
   an ample
   line bundle
   $L=\O_X(a_1F_1+a_2F_2+a_3\Delta+a_4\Sigma$).
   Using Proposition~\ref{Nabcd intersection numbers} one finds
   $$
      L\cdot\Nabcd=\frac 1D\,Q(a,b,c,d) \ ,
   $$
   where $Q$ is the quadratic form
   \begin{equation}\label{Q}\arraycolsep=0.25\arraycolsep
      Q(a,b,c,d)=\matr{a & b & c & d}
      \matr{a_1+a_3+a_4 & 0 & -a_3 & -a_4 \\
         0 & a_1+a_3+a_4 & a_4 & -a_3 \\
         -a_3 & a_4 & a_2+a_3+a_4 & 0 \\
         -a_4 & -a_3 & 0 & a_2+a_3+a_4 }
      \matr{a \\ b \\ c \\ d} \ .
   \end{equation}
   A computation shows that $Q$ is positive
   definite and of discriminant
   $$
      \delta=(L^2/2)^2 \ .
   $$

   We can now prove:

\begin{theorem}\label{computed by elliptic 2}
   Let $E_1$ be the elliptic curve with automorphism $[x]\mapsto
   [ix]$, and let $X=E_1\times E_1$. For any ample line bundle
   $L$ on $X$, the Seshadri constant $\eps(L)$ is computed by an
   elliptic curve.
\end{theorem}

   In the proof
   we will make use of
   the following result from the the geometry
   of numbers (see \cite{GL}, Chapter~6).

\begin{theorem}[Mahler]\label{thm Mahler}
   Let $Q$ be a positive definite quadratic form of four
   variables with discriminant $\delta$. Then there is a non-zero
   point $p\in\Z^4$ such that
   $$
      Q(p)\le\sqrt 2\sqrt[4]{\delta} \ .
   $$
\end{theorem}

\begin{proof}[Proof of Theorem~\ref{computed by elliptic 2}]
   Let $L=\O_X(a_1F_1+a_2F_2+a_3\Delta+a_4\Sigma)$ be an ample
   line bundle. We are interested in the minimum of the
   intersection numbers $L\cdot\Nabcd$ of $L$ with all elliptic
   curves $\Nabcd$. If the g.c.d.~$D$ that is associated with
   $(a,b,c,d)$ in \eqnref{gcd} is greater than one, then Lemma
   \ref{reduction lemma} (in the appendix) implies that the
   numbers $a,b,c,d$ my be replaced by numbers $\bar a,\bar
   b,\bar c,\bar d$ such that the corresponding g.c.d.~$\bar D$
   equals one, without altering the intersection product
   $L\cdot\Nabcd$ in the process. The upshot of this argument is
   that the intersection product $L\cdot\Nabcd$ may be minimized
   by taking the minimum of $Q$.

   Now, by
   Theorem~\ref{thm Mahler} there are integers $a,b,c,d$, not all
   of them zero, such that
   $$
      L\cdot N_{a,b,c,d}\le\sqrt
      2\,\sqrt[4]{\left(\frac{L^2}2\right)^2}=\sqrt{L^2} \ .
   $$
   To complete the proof, it therefore
   remains to show that there cannot be a
   curve of genus $>1$ computing $\eps(L)$.
   So suppose by way of contradiction
   that there is a submaximal curve $C$ for $L$ that is not
   elliptic. Since a non-elliptic curve on an abelian surface is
   automatically ample, we see from
   Proposition~\ref{prop ample submaximal} that $C$ is then
   submaximal for $\O_X(C)$ as well.
   On the other hand, applying to $\O_X(C)$
   the argument that we applied to $L$ at the
   beginning of the proof, we find that
   there is an elliptic curve $N$ with
   $$
      C\cdot N\le\sqrt{C^2} \ .
   $$
   But then, by Lemma~\ref{useful lemma}, $N$ would have to be a
   component of $C$, and this is a contradiction.
\end{proof}

   Our second aim in this section is to explicitly determine the
   Seshadri constants for all ample line bundles on $X$, i.e., to
   provide an analogue of Theorem~\ref{SC theorem}. It seems
   difficult to achieve this using the same methods that we
   applied in Section~\ref{sect without cx mult}. First, the
   increased number of variables makes it hard to derive direct
   estimates. Secondly, the analogue of Lemma~\ref{permutations}
   is not true, i.e., the generators of $\NS(X)$ may not be
   interchanged in arguments involving intersection numbers. For
   these reasons we proceed along a different path here, using
   a little elementary real analysis to obtain the desired bounds.

   Let us fix notation for the following lemma. If $M$ is a
   subset of $\R^n$, then $U_i(M)$ will denote the set of all
   points
   $(\liste x1n)\in\R^n$ such that there is an $(\liste m1n)\in M$ satisfying
   $\abs{x_i-m_i}\le 1$.

\begin{lemma}\label{analysis lemma}
   Let $f:\R^n\to\R$ be a partially differentiable function. Then
   the points,
   at which the restricted function $f\restr{\Z^n}$ is minimal,
   lie in the intersection
   $$
       \bigcap_{i=1}^n \
       U_i\Big(\Big\{x\in\R^n\with\pd f{x_i}(x)=0\Big\}\Big)
       \ .
   $$
\end{lemma}

\begin{proof}
   Suppose that $f\restr{\Z^n}$ is minimal at $m=(m_1,\dots,m_n)\in\Z^n$. Then
   $$
      f(m)\le f(m_1-1,m_2,\dots,m_n)
      \midtext{and}
      f(m)\le f(m_1+1,m_2,\dots,m_n) \ ,
   $$
   hence the function $t\mapsto f(t,m_2,\dots,m_n)$ assumes
   a
   local minimum at some point $t_1$ of
   the interval $[m_1-1,m_1+1]$. The partial derivative of
   $f$ vanishes then at $(t_1,m_2,\dots,m_n)$, which just means
   that $m$ is contained in the set
   $$
       U_1\Big(\Big\{x\in\R^n\with\pd f{x_i}(x)=0\Big\}\Big) \ .
   $$
   The analogous statement holds for $i=2,\dots,n$.
\end{proof}

   We are now ready to prove:

\begin{theorem}\label{SC theorem cx mult}
   Let $E_1$ be the elliptic curve admitting the automorphism
   $[x]\mapsto [ix]$, i.e., $E=\C/(\Z+i\Z)$, and let
   $L=\O_X(a_1F_1+a_2F_2+a_3\Delta+a_4\Sigma)$ be any ample
   line bundle
   on the abelian surface $X=E_1\times E_1$. Then
   \be
      && \eps(L) =
      \!\!\!\!
      \min\twoconditions{a,b,c,d\in\Z}{\abs a,\abs b,\abs c, \abs d\le B}
      \!\!\!\!
         \Big\{a_1(a^2+b^2)+a_2(c^2+d^2)
         \\[-3ex]
         && \hspace{10em} {}+a_3((a-c)^2+(b-d)^2)
         +a_4((a-d)^2+(b+c)^2)\Big\}
   \ee
   where
   $$
      B \eqdef\frac{8\max\set{\abs{a_1+a_3+a_4}^2,\abs{a_3}^2,\abs{a_4}^2,\abs{a_2+a_3+a_4}^2}}
         {a_1a_2+a_1a_3+a_1a_4+a_2a_3+a_2a_4+2a_3a_4} \ .
   $$
\end{theorem}

   As shown in Table~\ref{table eps cx mult}, the theorem can be used
   to effectively compute Seshadri constants from the
   coefficients $a_1,a_2,a_3,a_4$ of the line bundle.

\begin{table}
   \centering\footnotesize

\begin{tabular}{cccc@{\qquad}cccl}
$a_1$ & $a_2$ & $a_3$ & $a_4$ & $L^2$ & $\sqrt{L^2}$ & $\eps(L)$ & curves computing $\eps(L)$ \\ \tline
1 & 1 & 1 & 1 & 14 & $\approx 3,74$ & 3 & $ F_1, F_2$ \\
1 & 1 & 0 & 0 & 2 & $\approx 1,41$ & 1 & $ F_1, F_2$ \\
2 & 1 & 0 & 0 & 4 & 2 & 1 & $ F_1$ \\
0 & 0 & 1 & 1 & 4 & 2 & 2 & $F_1, F_2, \Delta, \Sigma, N_{ 1, 1, 0, 1}, N_{ 1, 0, 1, 1} $ \\
1 & 0 & 1 & 1 & 8 & $\approx 2,83$ & 2 & $F_1$ \\
1 & 1 & 1 & 0 & 6 & $\approx 2,45$ & 2 & $F_1, F_2, \Delta$ \\
2 & 2 & 1 & $-1$ & 4 & 2 & 2 & $F_1, F_2, \Delta, N_{ 1, 1, 1, 0}, N_{ 1, 0, 1, -1}, N_{ 1, 0, 0, -1} $ \\
$-1$ & 1 & 2 & 2 & 14 & $\approx 3,74$ & 3 & $ F_2, N_{ 1, 1, 0, 1}$ \\
$-1$ & 2 & 1 & 2 & 10 & $\approx 3,16$ & 2 & $ F_2$\\
4 & 4 & $-1$ & $-1$ & 4 & 2 & 2 & $F_1, F_2, N_{ 1, 1, 0, -1}, N_{ 1, 0, 0, -1}, N_{ 1, 0, -1, 0}, N_{ -1, 0, 1, 1} $\\
4 & 2 & 3 & $-2$ & 4 & 2 & 1 & $ N_{0, 1, 1, 1}$ \\
8 & 5 & $-1$ & $-2$ & 10 & $\approx 3,16$ & 2 & $F_1$ \\ \tline
\end{tabular}

   \caption{\label{table eps cx mult}
      Seshadri constants of the line bundles
      $L=\O_X(a_1F_1+a_2F_2+a_3\Delta+a_4\Sigma)$
      on $X=E_1\times E_1$
      }
\end{table}

\begin{proof}
   By the argument employed at the beginning of the proof of
   Theorem~\ref{computed by elliptic 2}, our task is to minimize
   the restriction $Q\restr{\Z^4}$. According to Lemma
   \ref{analysis lemma}, the points where this function is minimal
   lie in the intersection
   $\bigcap_{i=1}^4 U_i(\{x\in\R^n\with\pd Q{x_i}(x)=0\})$.
   We have for $x\in\R^4$
   $$
      \pd Q{x_1}(x)=2(a_1+a_3+a_4,0,-a_3,-a_4)\cdot x \ ,
   $$
   so the set of points in $\R^4$ whose first component has
   distance~1 from
   $\{\pd Q{x_1}=0\}$ is the union of the two affine
   hyperplanes
   $$
      \set{x\in\R^4\with (a_1+a_3+a_4,0,-a_3,-a_4)\cdot x
      = \pm (a_1+a_3+a_4) } \ ,
   $$
   and consequently
   $U_1(\{\pd Q{x_1}=0\})$
   is the set of points between these two hyperplanes.
   The intersection
   $\bigcap_{i=1}^4  U_i(\{\pd Q{x_i}=0\})$ is therefore a paralleloid,
   whose vertices are the solutions of the sixteen equations
   $$
      M\cdot x=\matr{
         \pm (a_1+a_3+a_4) \\
         \pm (a_1+a_3+a_4) \\
         \pm (a_2+a_3+a_4) \\
         \pm (a_2+a_3+a_4) } \ ,
   $$
   where $M$ is the matrix defining $Q$ in \eqnref{Q}.
   The lengths of these vertices, and therefore of all points in
   the paralleloid, are bounded from above by
   \be
      \norm x
      \le\norm{M\inverse}\cdot
      \norm{
         \matr{
            { a_1+a_3+a_4 } \\
            { a_1+a_3+a_4 } \\
            { a_2+a_3+a_4 } \\
            { a_2+a_3+a_4 } } } \ ,
   \ee
   and with a computation one finds that
   the right hand side is in turn bounded by the
   number $B$ defined in the statement of the theorem.
\end{proof}

\subsection{Complex multiplication by $e^{\pi i/3}$}

   We now turn to the elliptic curve $E_2$ with automorphism
   $\sigma:[x]\mapsto[e^{\pi i/3}x]$ and study the surface
   $X=E_2\times E_2$.
   A result analogous to Theorem \ref{SC theorem cx mult}
   holds in this case:

\begin{theorem}\label{SC theorem cx mult 2}
   Let $E_2$ be the elliptic curve admitting the automorphism
   $[x]\mapsto [e^{\pi i/3}x]$, i.e.,
   $E_2=\C / \Z+e^{\pi i/3}\Z$,
   and let
   $L=\O_X(a_1F_1+a_2F_2+a_3\Delta+a_4\Sigma)$ be any ample
   line bundle
   on the abelian surface $X=E_2\times E_2$. Then
   \be
      && \eps(L) =
      \!\!\!\!
      \min\twoconditions{a,b,c,d\in\Z}{\abs a,\abs b,\abs c, \abs d\le B}
      \!\!\!\!
         \Big\{a_1(a^2+ab+b^2)+a_2(c^2+cd+d^2)
         \\[-3ex]
         && \hspace{10em} {}+a_3((a-c)^2+(a-c)(b-d)+(b-d)^2)
         \\
         && \hspace{10em}
         {}+a_4((-a-b+d)^2+(-a-b+d)(b+c)+(b+c)^2)\Big\} \ ,
   \ee
   where
   $$
      B\eqdef\frac{8\max\set{\abs{2a_1+2a_3+2a_4}^2,\abs{2a_3+a_4}^2,\abs{a_3+2a_4}^2,\abs{a_3-a_4}^2,\abs{2a_2+2a_3+2a_4}^2}}
         {3(a_1a_2+a_1a_3+a_1a_4+a_2a_3+a_2a_4+a_3a_4)} \ .
   $$
\end{theorem}

   While the proof
   follows the same general strategy that we used for
   Theorem \ref{SC theorem cx mult},
   it is not totally analogous.
   In the remainder of this section we will
   indicate the course of the argument, mainly emphasizing
   the new aspects and formulas,
   without repeating arguments that
   can be adapted from the previous case.

   The automorphism $\sigma$ has the point $[0]$ as
   its only fixed point. The fibers $F_1,F_2$, the diagonal
   $\Delta$, and the graph $\Sigma$ of $\sigma$ generate the
   N\'eron-Severi group of $X$, and they have the
   intersection numbers
   $$
      F_1^2=F_2^2=\Delta^2=\Sigma^2=0
   $$
   and
   $$
      F_1\cdot F_2=F_1\cdot\Delta=F_2\cdot\Delta=F_1\cdot\Sigma=F_2\cdot\Sigma
      =\Delta\cdot\Sigma=1 \ .
   $$
   A line bundle $L=\O_X(a_1F_1+a_2F_2+a_3\Delta+a_4\Sigma)$ is
   ample if and only if
   \be
     & a_2+a_3+a_4=L\cdot F_1>0, \quad a_1+a_3+a_4=L\cdot F_2>0,\\
     & a_1+a_2+a_4=L\cdot\Delta>0,\quad a_1+a_2+a_3=L\cdot\Sigma>0,\\
     & 2(a_1a_2+a_1a_3+a_1a_4+a_2a_3+a_2a_4+a_3a_4)=L^2>0 \ .
   \ee
   As before, the elliptic curves on $X$ are given as
   the images $\Nabcd$ under suitable maps $E_2\times E_2\to X$.
   For their
   intersection numbers one obtains
   \be
      \displaystyle\Nabcd\cdot F_1&=&\frac{a^2+ab+b^2}D \\
      \displaystyle\Nabcd\cdot F_2&=&\frac{c^2+cd+d^2}D \\
      \displaystyle\Nabcd\cdot \Delta&=&\frac{(a-c)^2+(a-c)(b-d)+(b-d)^2}D \\
      \displaystyle\Nabcd\cdot \Gamma&=&\frac{(-a-b+d)^2+(-a-b+d)(b+c)+(b+c)^2}D
   \ee
   where one sets
   $$
      D \eqdef\gcd(a^2+ab+b^2,c^2+cd+d^2, ac+bc+bd,ad-bc) \ .
   $$
   In order to see this, one proves statements similar to
   Lemma~\ref{number of solutions} and
   Lemma~\ref{degree of sigma}.

   The next step is then to
   show that
   there is a submaximal elliptic curve for
   every ample line bundle on $X$.
   This is accomplished by considering the quadratic form
   $Q(a,b,c,d)$ given
   by the matrix
   $$
      \renewcommand\arraystretch{1.2}
      \matr{
         a_1+a_3+a_4 & \frac12(a_1+a_3+a_4) & \frac12(-2a_3-a_4) & \frac12(-a_3-2a_4) \\
         \frac12(a_1+a_3+a_4) & a_1+a_3+a_4 & \frac12(-a_3+a_4) & \frac12(-2a_3-a_4) \\
         \frac12(-2a_3-a_4) & \frac12(-a_3+a_4) & a_2+a_3+a_4 & \frac12(a_2+a_3+a_4) \\
         \frac12(-a_3-2a_4) & \frac12(-2a_3-a_4) & \frac12(a_2+a_3+a_4) & a_2+a_3+a_4
      }
   $$
   which governs the intersection numbers
   $L\cdot\Nabcd$.
   Finally, a minimization argument then leads to the
   estimates in Theorem~\ref{SC theorem cx mult 2}.
   A crucial auxiliary lemma that is needed for the proof
   (in the same way as
   Lemma~\ref{reduction lemma} is required for
   Theorem~\ref{SC theorem cx mult})
   is stated in the appendix as
   Lemma~\ref{reduction lemma 2}.

\section*{Appendix}
\renewcommand\thesection{A}
\setcounter{satz}{0}

   We state and prove here the elementary number-theoretic
   lemmas
   that are needed in the course of Sect.~\ref{sect cx mult}.

\begin{lemma}\label{solutions of congruences}
   Let $a,b,c,d$ be coprime integers, and let
   $$
      D = \gcd(a^2+b^2,c^2+d^2,ac+bd,ad-bc) \ .
   $$
   Then the system of congruences
   \begin{eqnarray}
       am-bn\equiv 0 \mod D \label{equiv1} \\
       bm+an\equiv 0 \mod D \label{equiv2} \\
       cm-dn\equiv 0 \mod D \label{equiv3} \\
       dm+cn\equiv 0 \mod D \label{equiv4}
   \end{eqnarray}
   has exactly $D$ solutions $(m,n)$ modulo $D$.
\end{lemma}

\begin{proof}
   (i) We first show that the system admits at most $D$
   solutions. As $a,b,c,d$ are coprime, we can write
   $\ell_1a+\ell_2b+\ell_3c+\ell_4d=1$
   with suitable integers $\ell_i$. The assertion follows then from
   the fact that
   for any solution $(m,n)$ we have
   $$
      m=(\ell_1a+\ell_2b+\ell_3c+\ell_4d)m
      \equiv(\ell_1b-\ell_2a-\ell_3d+\ell_4c)n \ .
   $$

   (ii) We claim next that a pair $(m,n)$ satisfying
   \eqnref{equiv1} and \eqnref{equiv3},
   automatically satisfies the
   remaining two congruences. In fact, we have
   \be
       && a(an+bm)=a^2n+abm\equiv (a^2+b^2)n\equiv 0 \\
       && b(an+bm)=abn+b^2n\equiv (a^2+b^2)m\equiv 0 \\
       && c(an+bm)=acn+bcm\equiv (ac+bd)n\equiv 0 \\
       && d(an+bm)=adn+bdm\equiv (ac+bd)m\equiv 0
   \ee
   and, as $a,b,c,d$ are coprime, it follows that $an+bm\equiv
   0$. The equivalence $cn+dm\equiv0$ follows in the analogous
   manner.

   (iii) We assert that $\gcd(a,D)=\gcd(b,D)$ and
   $\gcd(c,D)=\gcd(d,D)$.
   In fact,
   the number $A=\gcd(a,D)$ divides all of the numbers
   $a,a^2+b^2,ac+bd,ad-bc$, hence it also divides the numbers
   $ab,bb,cb,db$. The coprimality of $a,b,c,d$ then implies that
   $A$ divides $b$, and hence $\gcd(a,D)=\gcd(a,b,D)$.
   The analogous statements hold for
   $\gcd(b,D),\gcd(c,D),\gcd(d,D)$, and this implies the
   assertion.

   (iv) Finally, we show that for every integer $n$ there is an
   integer $m$ such that $(m,n)$ is a solution of \eqnref{equiv1}
   and \eqnref{equiv3}.
   Using (iii), we see that we have
   $\gensubgroup a=\gensubgroup{\gcd(a,D)}
   =\gensubgroup{\gcd(b,D)}=\gensubgroup b$ for
   the generated subgroups of $\Z/D\Z$. So, given $n$, the congruence
   \eqnref{equiv1} has
   exactly $A=\gcd(a,D)$ solutions, and they are of the form
   \begin{equation}\label{sols}
      \hat m,\ \hat m+\frac DA,\ \dots,\ \hat m+(A-1)\frac DA \ .
   \end{equation}
   Suppose that for two indices $\ell,k\in\{0,\dots,A-1\}$ there is
   an equivalence
   $c(\hat m+\ell\frac DA)-dn\equiv c(\hat m+k\frac DA)-dn$
   modulo $D$.
   Then $c(k-l)\frac DA\equiv 0$ modulo $D$, and hence
   \begin{equation}\label{equiv5}
      (k-l)\frac DA\equiv 0 \mod \frac D{\gcd(c,D)} \ .
   \end{equation}
   Let now
   $C=\gcd(c,D)$ and $B=\frac D{AC}$. As $A$ and $C$ are coprime,
   it follows that $B$ is an integer. Therefore \eqnref{equiv5}
   says that $(k-l)BC$ is a multiple of $AB$. But then $A$
   divides $k-\ell$, which implies $k=\ell$.
   So we have shown that
   when $m$ runs through the $A$ solutions \eqnref{sols}, the
   numbers $cm-dn$ are distinct modulo $D$.

   On the other hand, as $D$ divides $ad-bc$, we have
   $a(cm-dn)\equiv c(am-bn)$ mod $D$,
   hence $cm-dn\equiv 0$ mod $\frac DA$.
   This leaves only $A$
   possible values modulo $D$ for the expression $cm-dn$,
   namely the multiples of $\frac DA$.
   We infer that each of these $A$ values must appear, among them
   the value $0$. The corresponding pair $(m,n)$ is then a
   solution as required.
\end{proof}

\begin{lemma}\label{reduction lemma}
   Let $a,b,c,d$ be coprime integers, and let
   $$
      D = \gcd(a^2+b^2,c^2+d^2,ac+bd,ad-bc) \ .
   $$
   Then there are coprime integers $\bar a,\bar b,\bar c,\bar d$
   such that
   $$
      \gcd(\bar a^2+\bar b^2,\bar c^2+\bar d^2,\bar a\bar c+\bar b\bar d,\bar a\bar d-\bar b\bar c)=1
   $$
   and
   \be
      \bar a^2+\bar b^2=\frac 1D(a^2+b^2),&& \bar c^2+\bar d^2=\frac 1D(c^2+d^2), \\
      \bar a\bar c+\bar b\bar d=\frac 1D(ac+bd),&& \bar a\bar d-\bar b\bar c=\frac 1D(ad-bc) \ .
   \ee
\end{lemma}

\begin{proof}
   The idea is to consider matrices
   $M=\bigl({\alpha\atop -\beta}{\beta\atop \alpha}\bigr)$ over
   $\R$ such that the two image vectors
   \begin{equation}\label{images integral}\arraycolsep=0.25\arraycolsep
      \matr{\bar a\\ \bar b}\eqdef M\cdot\matr{a\\ b}=\matr{\alpha a+\beta b\\ -\beta a+\alpha b}
      \midtext{and}
      \matr{\bar c\\ \bar d}\eqdef M\cdot\matr{c\\ d}=\matr{\alpha c+\beta d\\ -\beta c+\alpha d}
   \end{equation}
   are integral. If $M$ is such a matrix, then
   $(a^2+b^2)\alpha=a(\alpha a+\beta b)+b(-\beta a+\alpha
   b)\in\Z$
   and similarly $(c^2+d^2)\alpha\in\Z$,
   $(ac+bd)\alpha\in\Z$, and
   $(ad-bc)\alpha\in\Z$. Consequently $\alpha$, and for the same
   reason $\beta$, are necessarily of the form
   $\alpha=\frac xD$, $\beta=\frac yD$ for some integers $x,y$.
   The conditions \eqnref{images integral} are then equivalent to
   \be
       ax+by\equiv 0 \mod D\ , &&
       bx-ay\equiv 0 \mod D\ , \\
       cx+dy\equiv 0 \mod D\ , &&
       dx-cy\equiv 0 \mod D\ .
   \ee
   Now, the proof of Lemma \ref{solutions of congruences} shows
   that this system is solvable even with a prescribed value for
   $x$. Let then $y$ be the solution associated with $x=1$.
   We have modulo $D$ the equivalence
   $0\equiv x(ax+by)-y(bx-ay)=a(x^2+y^2)$,
   and similarly
   $0\equiv b(x^2+y^2)$, as well as
   $0\equiv c(x^2+y^2)$ and
   $0\equiv d(x^2+y^2)$.
   As $a,b,c,d$ are coprime, this implies $x^2+y^2\equiv 0$,
   i.e., $\frac{x^2+y^2}D$ is an integer.
   We find
   \be
      \bar a^2+\bar b^2=\frac{x^2+y^2}{D^2}(a^2+b^2), &&
      \bar c^2+\bar d^2=\frac{x^2+y^2}{D^2}(c^2+d^2), \\
      \bar a\bar c+\bar b\bar d=\frac{x^2+y^2}{D^2}(ac+bd), &&
      \bar a\bar d-\bar b\bar c=\frac{x^2+y^2}{D^2}(ad-bc). \\
   \ee
   The number
   $\bar D\eqdef
   \gcd(\bar a^2+\bar b^2,\bar c^2+\bar d^2,\bar a\bar c+\bar b\bar d,\bar a\bar d-\bar b\bar c)$
   satisfies
   $$
      \bar D=\frac{x^2+y^2}{D^2}D\le\frac{1+(D-1)^2}{D^2}D \ ,
   $$
   where the right hand side is smaller than $D$ if $D>1$.
   We can now repeat the argument until eventually $\bar D=1$.
\end{proof}

   The following lemma can be proven
   using similar arguments.
   We leave the details to the reader.

\begin{lemma}\label{reduction lemma 2}
   Let $a,b,c,d$ be coprime integers, and let
   $$
      D=\gcd(a^2+ab+b^2,c^2+cd+d^2, ac+bc+bd,ad-bc) \ .
   $$
   Then there are coprime integers $\bar a,\bar b,\bar c,\bar d$
   such that
   $$
      \gcd(\bar a^2+\bar a\bar b+\bar b^2,\bar c^2+\bar c\bar d+\bar d^2, \bar a\bar c+\bar b\bar c+\bar b\bar d,\bar a\bar d-\bar b\bar c)=1 \ .
   $$
   and
   \be
      \bar a^2+\bar a\bar b+\bar b^2=\frac 1D(a^2+ab+b^2),
      &&
      \bar c^2+\bar c\bar d+\bar d^2=\frac 1D(c^2+cd+d^2),
      \\
      \bar a\bar c+\bar b\bar c+\bar b\,\bar d=\frac 1D(ac+bc+bd),
      &&
      \bar a\bar d-\bar b\bar c=\frac 1D(ad-bc) \ .
   \ee
\end{lemma}

\bigskip
\footnotesize
   Tho\-mas Bau\-er, Christoph Schulz,
   Fach\-be\-reich Ma\-the\-ma\-tik und In\-for\-ma\-tik,
   Philipps-Uni\-ver\-si\-t\"at Mar\-burg,
   Hans-Meer\-wein-Stra{\ss}e,
   D-35032~Mar\-burg, Germany.

\nopagebreak
   E-mail: \texttt{tbauer@mathematik.uni-marburg.de},
   \texttt{schulz@mathematik.uni-marburg.de}

\end{document}